\documentclass[12pt,reqno]{amsart}
\usepackage{amssymb,amsmath,hyperref,amsthm}
\newcommand{\sg}{\textnormal{sg}}

\newtheorem{theorem}{Theorem}
\newtheorem{lemma}[theorem]{Lemma}

\newtheorem{proposition}[theorem]{Proposition}

\theoremstyle{remark}

\theoremstyle{definition}

\numberwithin{theorem}{section} \numberwithin{equation}{section}
\numberwithin{example}{section}

\title{Hecke-Rogers double-sums and false theta functions}

\author{Eric T. Mortenson}

\begin{document}

\date{20 October 2020}



\begin{abstract}
We develop a setting in which one can evaluate certain Hecke-Rogers series in terms of false theta functions.  We apply our setting to recent false theta function identities of Chan and Kim as well as Andrews and Warnaar.
\end{abstract}

\address{Saint Petersburg State University, Department of Mathematics and Computer Science, Saint Petersburg, 199178, Russia}
\email{etmortenson@gmail.com}
\maketitle
\setcounter{section}{-1}
%\setcounter{tocdepth}{1} % omits subsections from TOC

%\tableofcontents

\section{Notation}\label{section:double-type-II}

Let $q$ be a complex number, $0<|q|<1$, and define $\mathbb{C}^*:=\mathbb{C}-\{0\}$.  We recall:
\begin{gather*}
(x)_n=(x;q)_n:=\prod_{i=0}^{n-1}(1-q^ix), \ \ (x)_{\infty}=(x;q)_{\infty}:=\prod_{i= 0}^{\infty}(1-q^ix),\\
 j(x;q):=(x)_{\infty}(q/x)_{\infty}(q)_{\infty}=\sum_{n=-\infty}^{\infty}(-1)^nq^{\binom{n}{2}}x^n.
\end{gather*}
where in the last line the equivalence of product and sum follows from Jacobi's triple product identity.    The following are special cases of the above definition.  Let $a$ and $m$ be integers with $m$ positive.  Define
\begin{gather*}
J_{a,m}:=j(q^a;q^m), \ \ \overline{J}_{a,m}:=j(-q^a;q^m), \ {\text{and }}J_m:=J_{m,3m}=\prod_{i= 1}^{\infty}(1-q^{mi}).
\end{gather*}
We will use the following definition of an Appell-Lerch sum.   Let $x,z\in\mathbb{C}^*$ with neither $z$ nor $xz$ an integral power of $q$. Then
\begin{equation}
m(x,q,z):=\frac{1}{j(z;q)}\sum_{r=-\infty}^{\infty}\frac{(-1)^rq^{\binom{r}{2}}z^r}{1-q^{r-1}xz}.
\label{equation:mxqz-def}
\end{equation}

\section{Introduction}

In previous work with Hickerson \cite{HM}, we found a formula that expands Hecke-Rogers double-sums in terms of theta functions and Appell-Lerch functions.  The expansions yield classical modular and mock modular Hecke-Rogers double-sum identities as straightforward consequences.  In particular, our double-sum expansions \cite{HM} give a new proof of the mock theta conjectures \cite{H1, H2}.  In our present work, we obtain a new formula that expands double-sums of a different symmetry type in terms of false theta functions, and use the expansion to give new proofs of recent false theta function identities of Chan and Kim \cite{CK}.  We will also demonstrate how our setting applies to the false theta function identities of Andrews and Warnaar \cite{AW}.  We begin by reviewing the results of \cite{HM}.

We begin with Hecke-Rogers series.  Let $x,y\in\mathbb{C}^*$, then
\begin{equation}
f_{a,b,c}(x,y,q):=\Big ( \sum_{r,s\ge 0}-\sum_{r,s<0}\Big ) (-1)^{r+s}x^ry^sq^{a\binom{r}{2}+brs+c\binom{s}{2}},\label{equation:fabc-def}
\end{equation}
where $a,c>0$ and $b>-\sqrt{ac}$.  Hecke \cite{He} conducted the first systematic study of such double-sums, but special cases appeared earlier in work of Rogers \cite{R}.

By using a heuristic \cite[Section $3$]{HM} that relates  Appell--Lerch functions (\ref{equation:mxqz-def}) to divergent partial theta functions, we were led to expressions such as
\begin{align}
f_{1,2,1}(x,y,q)&=j(y;q)m\big (\frac{q^2x}{y^2},q^3,-1\big )+j(x;q)m\big (\frac{q^2y}{x^2},q^3,-1\big )
\label{equation:f121}\\
&\ \ \ \ \ - \frac{yJ_3^3j(-x/y;q)j(q^2xy;q^3)}
{\overline{J}_{0,3}j(-qy^2/x;q^3)j(-qx^2/y;q^3)},\notag
\end{align}
which has the immediate corollaries:
\begin{align}
f_{1,2,1}(q,q,q)&=J_{1}^2,\label{equation:f121-ex1}\\
f_{1,2,1}(q,-q,q)&=\overline{J}_{1,4}\phi(q),\label{equation:f121-ex2}
\end{align}
where
\begin{equation}
\phi(q):=\sum_{n=0}^{\infty}\frac{(-1)^nq^{n^2}(q;q^2)_{n}}{(-q)_{2n}}=2m(q,q^3,-1)
\end{equation}
is a sixth order mock theta function \cite{AH}.   We point out that (\ref{equation:f121}) is a special case of a more general expansion \cite[Theorem 1.3]{HM}.

In this note, we explore Hecke-Rogers double-sums of a different symmetry type; in particular we will study sums of the form
\begin{equation}
g_{a,b,c}(x,y,q):=\Big ( \sum_{r,s\ge 0}+\sum_{r,s<0}\Big )(-1)^{r+s}x^ry^sq^{a\binom{r}{2}+brs+c\binom{s}{2}},
\label{equation:gabc-def}
\end{equation}
where $a,c>0$, $b>-\sqrt{ac}$.  Well-known functions have the form (\ref{equation:gabc-def}).  For example, we have the function studied  Andrews, Dyson, and Hickerson  \cite{ADH} and Cohen \cite{C}
\begin{equation}
\sigma(q)
:=1+\sum_{n=1}^{\infty}\frac{q^{n(n+1)/2}}{(-q;q)_n}
=g_{1,3,3}(-q,q^2,q)-qg_{1,3,3}(-q^3,q^4,q),\label{equation:ADH-g131}
\end{equation}
as well as recent recent results of Chan and Kim on false theta functions \cite{CK}.

We will give new proofs of several false theta function identities found in Chan and Kim \cite{CK}.  In particular we will prove Warnaar's \cite{CK, W1}
\begin{equation}
g_{1,2,2}(q,-q^3,q)=1+2\sum_{n=1}^{\infty}(-1)^nq^{n(n+1)/2},\label{equation:Warnaar}
\end{equation}
 as well as Chan and Kim's \cite{CK}
\begin{align}
g_{2,2,2}(-q^2,-q^3,q)&=\frac{1}{1-q},\label{equation:Prop3pt4A}\\
g_{2,2,2}(q^2,-q^3,q)&=\frac{1}{1+q}\Big ( 1+2\sum_{n= 1}^{\infty}(-1)^nq^{n(n+1)}\Big ),\label{equation:Prop3pt4B}\\
g_{1,2,4}(-q,-q^4,q)&=\frac{1}{1-q},\label{equation:Thm1pt1A}\\
g_{1,2,4}(q,-q^4,q)&=\frac{1}{1-q}\Big (-1+2\sum_{n=0}^{\infty}(-1)^nq^{n(n+1)/2}\Big ),\label{equation:Thm1pt1B}
\end{align}
where Identities (\ref{equation:Prop3pt4A}) and (\ref{equation:Prop3pt4B}) are from \cite[Proposition $3.4$]{CK}, and Identities (\ref{equation:Thm1pt1A}) and (\ref{equation:Thm1pt1A}) are from \cite[Theorem $1.1$]{CK}.  Identity (\ref{equation:Thm1pt1B}) has been slightly rewritten.  We also prove \cite[Theorem $3.6$]{CK} 
\begin{equation}
q^4G(q^3)=q-\sum_{n=1}^{\infty}\Big ( \frac{n}{3}\Big )q^{n^2}-2\sum_{n=0}^{\infty}q^{9n^2+15n+7}(1-q^{6n+6}), \label{equation:Thm3pt6}
\end{equation}
where $(\tfrac{n}{3})$ is the Legendre Symbol and
\begin{equation}
G(q):=g_{6,3,2}(-q^5,-q^3,q)-qg_{6,3,2}(-q^7,-q^4,q).
\end{equation}

In \cite{AW}, Andrews and Warnaar prove several false theta function identities using the Bailey Transform.   In the course of their proofs, they write many of the $q$-hypergeometric series in terms of (\ref{equation:gabc-def}); however, this is not stated explicitly.  For example, in the course of proving \cite[(1.1a), (1.1c), (1.1d)]{AW}
\begin{align}
\sum_{n=0}^{\infty}\frac{(-1)^nq^{n(n+1)}(q;q^2)_{n}}{(-q;q)_{2n+1}}
&=\sum_{n=0}^{\infty}(-1)^nq^{n(n+1)/2},\label{equation:AW-1pt1a}\\
\sum_{n=0}^{\infty}\frac{(q;q^2)_{n}q^n}{(-q;q)_{2n+1}}
&=\sum_{n=0}^{\infty}(-1)^nq^{3n(n+1)/2},\label{equation:AW-1pt1c}\\
\sum_{n=0}^{\infty}\frac{(q;-q)_{2n}q^n}{(-q;q)_{2n+1}}
&=\sum_{n=0}^{\infty}(-1)^nq^{2n(n+1)},\label{equation:AW-1pt1d}
\end{align}
they prove as intermediate steps
\begin{align}
\sum_{n=0}^{\infty}\frac{(-1)^nq^{n(n+1)}(q;q^2)_{n}}{(-q;q)_{2n+1}}
&=g_{1,3,1}(q,q^2,q),\label{equation:AW-1pt1a-g131}\\
\sum_{n=0}^{\infty}\frac{(q;q^2)_{n}q^n}{(-q;q)_{2n+1}}
&=g_{3,1,3}(q^2,q^3,q),\label{equation:AW-1pt1c-g313}\\
\sum_{n=0}^{\infty}\frac{(q;-q)_{2n}q^n}{(-q;q)_{2n+1}}
&=g_{1,0,1}(q^2,q^4,q^4).\label{equation:AW-1pt1d-g404}
\end{align}
We will also show how the Hecke-Rogers series found in \cite{AW} fit within the context of this note.

In Section \ref{section:prelim}, we prove properties for (\ref{equation:gabc-def}) that are analogous to  those for (\ref{equation:fabc-def}).  The new properties will be what proves the false theta function identities.  The reader may have noticed that although the examples (\ref{equation:f121-ex1}) and (\ref{equation:f121-ex2}) of (\ref{equation:f121}) are weight one, the series (\ref{equation:gabc-def}) in this note usually evaluate to weight one-half false theta functions.  In Section \ref{section:remarks}, we demonstrate how to obtain weight one-half objects  from (\ref{equation:fabc-def}).

In Sections \ref{section:chanKim} and \ref{section:andrewsWarnaar}, we give new proofs of false theta function identities found in \cite{CK} and \cite{AW} respectively.  In Section \ref{section:signFlips}, we give examples in which sign flips in expressions involving (\ref{equation:fabc-def}) and (\ref{equation:gabc-def}) result in false theta functions.   In particular, we will show a simple sign flip takes the function studied by Andrews, Dyson, and Hickerson (\ref{equation:ADH-g131}) to a false theta function
\begin{equation}
g_{1,3,3}(-q,q^2,q)+qg_{1,3,3}(-q^3,q^4,q)
=\sum_{s=-\infty}^{\infty}\sg(s)(-1)^sq^{s(3s+1)/2}.
\end{equation}

\section{Preliminaries} \label{section:prelim}
In this section we obtain $g_{a,b,c}(x,y,q)$ analogs of $f_{a,b,c}(x,y,q)$ properties.   We recall some basic facts about Hecke-Rogers series.   We have \cite[$(1.14)$]{H2}
\begin{equation}
\sum_{\sg(r)=\sg(s)}\sg(r)c_{r,s}
=-\sum_{\sg(r)=\sg(s)}\sg(r)c_{-1-r,-1-s},\label{equation:H2eq1.14}
\end{equation}
and \cite[$(1.15)$]{H2}
\begin{equation}
\sum_{\sg(r)=\sg(s)}\sg(r)c_{r,s}= \sum_{\sg(r)=\sg(s)}\sg(r)c_{r+\ell,s+k}+\sum_{r=0}^{\ell-1}\sum_{s\in \mathbb{Z}}c_{r,s}+\sum_{s=0}^{k-1}\sum_{r\in \mathbb{Z}}c_{r,s},\label{equation:H2eq1.15}
\end{equation}
where 
\begin{equation}
\sg(r):=
\begin{cases}
1 & r\ge 0,\\
-1 & r<0.
\end{cases}
\end{equation}

When $b<a$ we adopt the convention that 
\begin{equation}
\sum_{r=a}^{b} c_r:=-\sum_{r=b+1}^{a-1} c_r,\label{equation:sumconvention}
\end{equation}
which has the useful consequence
\begin{equation}
\sum_{r=0}^{-1}c_r=-\sum_{r=0}^{-1}c_r=0.
\end{equation}

Setting 
\begin{equation}
c_{r,s}:=(-1)^{r+s}x^ry^sq^{a\binom{r}{2}+brs+c\binom{s}{2}}
\end{equation}
 in (\ref{equation:H2eq1.14}) and (\ref{equation:H2eq1.15}) results in the following two propositions:
 \begin{proposition} \cite[Proposition 6.2]{HM} For $x,y\in\mathbb{C}^*$
\begin{equation}
f_{a,b,c}(x,y,q)=-\frac{q^{a+b+c}}{xy}f_{a,b,c}(q^{2a+b}/x,q^{2c+b}/y,q).\label{equation:f-flip}
\end{equation}
\end{proposition}
\begin{proposition}\cite[Proposition 6.3]{HM}  \label{proposition:f-functionaleqn} For $x,y\in\mathbb{C}^*$ and $\ell, k \in \mathbb{Z}$
\begin{align}
f_{a,b,c}(x,y,q)&=(-x)^{\ell}(-y)^kq^{a\binom{\ell}{2}+b\ell k+c\binom{k}{2}}f_{a,b,c}(q^{a\ell+bk}x,q^{b\ell+ck}y,q) \label{equation:f-shift}\\
&\ \ \ \ +\sum_{m=0}^{\ell-1}(-x)^mq^{a\binom{m}{2}}j(q^{mb}y;q^c)+\sum_{m=0}^{k-1}(-y)^mq^{c\binom{m}{2}}j(q^{mb}x;q^a).\notag
\end{align}
\end{proposition}

We now state two analogous identities which appear to be new.
\begin{proposition} \label{proposition:g-shift} We have
\begin{align}
\sum_{\sg(r)=\sg(s)}c_{r,s}
&= \sum_{\sg(r)=\sg(s)} c_{-1-r,-1-s},\label{equation:H2eq1.14g}\\
\sum_{\sg(r)=\sg(s)}c_{r,s}
&= \sum_{\sg(r)=\sg(s)}c_{r+\ell,s+k}
\label{equation:H2eq1.15g}\\
&\ \ \ \ \ +\sum_{r=0}^{\ell-1}\sum_{s\in \mathbb{Z}}\sg(s)c_{r,s}
+\sum_{s=0}^{k-1}\sum_{r\in \mathbb{Z}}\sg(r)c_{r,s} -2\sum_{r=0}^{\ell-1}\sum_{s=0}^{k-1}c_{r,s}.
\notag
\end{align}
\end{proposition}

Setting 
\begin{equation}
c_{r,s}:=(-1)^{r+s}x^ry^sq^{a\binom{r}{2}+brs+c\binom{s}{2}}
\end{equation}
in (\ref{equation:H2eq1.14g}) and (\ref{equation:H2eq1.15g}) results in the following two propositions:
\begin{proposition}\label{proposition:g-flip} For $x,y\in\mathbb{C}^{*}$
\begin{equation}
g_{a,b,c}(x,y,q)=\frac{q^{a+b+c}}{xy}g_{a,b,c}(q^{2a+b}/x,q^{2c+b}/y,q).\label{equation:g-flip}
\end{equation}
\end{proposition}
\begin{proposition}\label{theorem:g-shift} For $x,y\in\mathbb{C}^{*}$ and $\ell, k \in \mathbb{Z}$
\begin{align}
g_{a,b,c}&(x,y,q)=(-x)^{\ell}(-y)^kq^{a\binom{\ell}{2}+b\ell k+ c\binom{k}{2}}g_{a,b,c}(q^{a\ell+bk}x,q^{b\ell + ck}y,q)\label{equation:g-shift}\\
& +\sum_{r=0}^{\ell-1}(-x)^{r}q^{a\binom{r}{2}}\sum_{s\in\mathbb{Z}}\sg(s)(-y)^sq^{brs}q^{c\binom{s}{2}}
+\sum_{s=0}^{k-1}(-y)^{s}q^{c\binom{s}{2}}\sum_{r\in\mathbb{Z}}\sg(r)(-x)^rq^{brs}q^{a\binom{r}{2}} \nonumber\\
& -2\sum_{r=0}^{\ell-1}\sum_{s=0}^{k-1}(-x)^{r}(-y)^{s}q^{a\binom{r}{2}+brs+c\binom{s}{2}}.\nonumber
\end{align}
\end{proposition}

\begin{proof}[Proof of Proposition \ref{proposition:g-shift}]  The first identity follows from fact that 
\begin{equation*}
\sg(r)=-\sg(-1-r).
\end{equation*}
For the second identity, a direct calculation yields
{\allowdisplaybreaks \begin{align*}
\sum_{\sg(r)=\sg(s)}c_{r,s}-\sum_{\sg(r)=\sg(s)}c_{r+\ell,s+k}
&=\sum_{\sg(r)=\sg(s)}c_{r,s}-\sum_{\sg(r-\ell)=\sg(s-k)}c_{r,s}\\
&=\Big ( \sum_{\substack{r\ge 0 \\s\ge 0}}+ \sum_{\substack{r< 0 \\s< 0}}
-\sum_{\substack{r-\ell\ge 0 \\s-k\ge 0}}- \sum_{\substack{r-\ell< 0 \\s-k< 0}}\Big ) c_{r,s}\\
&=\Big [ \sum_{\substack{r\ge 0 \\s\ge 0}}-\sum_{\substack{r-\ell\ge 0 \\s-k\ge 0}}
-\Big (  \sum_{\substack{r-\ell< 0 \\s-k< 0}}- \sum_{\substack{r< 0 \\s< 0}}\Big ) \Big ] c_{r,s}\\
&=\Big [ \sum_{\substack{r\ge 0 \\ 0\le s < k}}+\sum_{\substack{0\le r< \ell \\ s\ge 0 }}
-\sum_{\substack{0\le r< \ell \\ 0\le s < k }}
 -\Big (  \sum_{\substack{r< \ell \\0\le s < k}} +  \sum_{\substack{0\le  r <\ell \\ s<k}} 
-\sum_{\substack{0\le r< \ell \\ 0\le s < k }}\Big ) \Big ]c_{r,s}\\
&=\Big [ \sum_{\substack{r\ge 0 \\ 0\le s < k}}+\sum_{\substack{0\le r< \ell \\ s\ge 0 }}
-\sum_{\substack{0\le r< \ell \\ 0\le s < k }}
 -\Big (  \sum_{\substack{0\le r< \ell \\0\le s < k}} +\sum_{\substack{r< 0 \\0\le s < k}} 
 +  \sum_{\substack{0\le  r <\ell \\ s<0}} \Big ) \Big ]c_{r,s}\\
&=\sum_{r=0}^{\ell-1}\sum_{s\in\mathbb{Z}}\sg(s)c_{r,s}
+\sum_{s=0}^{k-1}\sum_{r\in\mathbb{Z}}\sg(r)c_{r,s}
-2\sum_{r=0}^{\ell-1}\sum_{s=0}^{k-1}c_{r,s}.\qedhere
\end{align*}}%
\end{proof}

%%%%% 
%%%%%
%%%%%

\section{Remarks on weight one-half double-sums}\label{section:remarks}
In this section we demonstrate by example that one can obtain weight one-half theta functions from (\ref{equation:fabc-def})  via both (\ref{equation:f-shift}) and (\ref{equation:f121}).  We take a specialization of (\ref{equation:f-shift})
{\allowdisplaybreaks \begin{align*}
f_{1,2,1}(-q,q,q^2)&=q^{\ell}(-q)^kq^{2\binom{\ell}{2}+4\ell k+2\binom{k}{2}}f_{1,2,1}(q^{2\ell+4k}x,q^{4\ell+2k}y,q^2) \\
&\ \ \ \ +\sum_{m=0}^{\ell-1}(-x)^mq^{2\binom{m}{2}}j(q^{4m}y;q^2)
+\sum_{m=0}^{k-1}(-y)^mq^{2\binom{m}{2}}j(q^{4m}x;q^2),
\end{align*}}%
and further specialize with $(\ell,k)=(1,1)$ to obtain
{\allowdisplaybreaks \begin{align*}
f_{1,2,1}(-q,q,q^2)&=-q^6f_{1,2,1}(-q^7,q^7,q^2) +j(q;q^2) +j(-q;q^2)\\
&=- f_{1,2,1}(-q,q,q^2)+j(q;q^2) +j(-q;q^2),
\end{align*}}%
where we have used (\ref{equation:f-flip}).   Comparing the extreme left with the extreme right gives
\begin{align*}
f_{1,2,1}(-q,q,q^2)&=\frac{1}{2}\Big ( j(q;q^2) +j(-q;q^2)\Big )
=j(-q^4;q^8),
\end{align*}
where we have used the fact \cite[(2.2f)]{HM}:
\begin{equation*}
j(z;q)=j(-qz^2;q^4)-zj(-q^3z^2;q^4).
\end{equation*}

Specializing (\ref{equation:f121}) we have
\begin{align*}
f_{1,2,1}(-q,q,q^2)&=j(q;q^2)m\big (-q^3,q^6,-1\big )+j(-q;q^2)m\big (q^3,q^6,-1\big )\\
& \ \ \ \ \ -\frac{1}{\overline{J}_{0,6}}\cdot \frac{qJ_{6}^3j(1;q^2)j(-q^6;q^6)}
{j(q^3;q^6)j(-q^3;q^6)}\\
&= j(q;q^2)\cdot\frac{1}{2} +j(-q;q^2)\cdot \frac{1}{2}+0,\\
&=j(-q^4;q^8),
\end{align*}
where we have used \cite[(3.3)]{HM}
\begin{equation*}
m(q,q^2,-1)=\frac{1}{2},
\end{equation*}
the fact $j(1;q^2)=0$, and again \cite[(2.2f)]{HM}.

%%%%% 
%%%%%
%%%%%

\section{False theta function identities of Chan and Kim}\label{section:chanKim}

\subsection{Proof of Identity (\ref{equation:Warnaar})}
We begin with a proposition
\begin{proposition}\label{proposition:prop1}
We have
\begin{align}
g_{1,2,2}(q,-q^3,q)&=(-q)^{\ell}q^{3k}q^{\binom{\ell}{2}+2\ell k + 2\binom{k}{2}}g_{1,2,2}(q^{\ell+2k+1},-q^{2\ell+2k+3},q)
\label{equation:prop1eq}\\
& \ \ \ \ \ -\sum_{r=0}^{\ell-1}(-1)^rq^{\binom{r+1}{2}} \sum_{s=1}^{2r+1}q^{s(s-2-2r)}
 +2\sum_{s=0}^{k-1}q^{s^2+2s} 
\sum_{r=0}^{\infty}(-1)^rq^{r(r+1+4s)/2}  \nonumber\\
& \ \ \ \ \ -2\sum_{r=0}^{\ell-1}\sum_{s=0}^{k-1}(-1)^{r}q^{r+3s+\binom{r}{2}+2rs+2\binom{s}{2}}.\nonumber
\end{align}
\end{proposition}

\begin{proof}[Proof of Proposition \ref{proposition:prop1}]
We take the appropriate specializations of (\ref{equation:g-shift}).  We have
{\allowdisplaybreaks \begin{align*}
g_{1,2,2}&(q,-q^3,q)-(-q)^{\ell}q^{3k}q^{\binom{\ell}{2}+2\ell k + 2\binom{k}{2}}g_{1,2,2}(q^{\ell+2k+1},-q^{2\ell+2k+3},q)\\
&=\sum_{r=0}^{\ell-1}(-1)^rq^{\binom{r+1}{2}}\sum_{s\in\mathbb{Z}}\sg(s)q^{s^2+2s+2rs}
 +\sum_{s=0}^{k-1}(q^3)^{s}q^{2\binom{s}{2}}\sum_{r\in\mathbb{Z}}\sg(r)(-1)^r(q^{2s}q)^rq^{\binom{r}{2}}\\
& \ \ \ \ \ -2\sum_{r=0}^{\ell-1}\sum_{s=0}^{k-1}(-1)^{r}q^{r+3s+\binom{r}{2}+2rs+2\binom{s}{2}}\\
&=\sum_{r=0}^{\ell-1}(-1)^rq^{\binom{r+1}{2}}\Big [ \sum_{s=0}^{\infty}q^{s(s+2+2r)}
-\sum_{s=1}^{\infty}q^{s(s-2-2r)}\Big ]\\
& \ \ \ \ \ +\sum_{s=0}^{k-1}q^{s^2+2s}\Big [ \sum_{r=0}^{\infty}(-1)^rq^{r(r+1+4s)/2}
-\sum_{r=1}^{\infty}(-1)^rq^{r(r-1-4s)/2}\Big ] \\
& \ \ \ \ \ -2\sum_{r=0}^{\ell-1}\sum_{s=0}^{k-1}(-1)^{r}q^{r+3s+\binom{r}{2}+2rs+2\binom{s}{2}}\\
& = -\sum_{r=0}^{\ell-1}(-1)^rq^{\binom{r+1}{2}} \sum_{s=1}^{2r+1}q^{s(s-2-2r)}\\
& \ \ \ \ \ +\sum_{s=0}^{k-1}q^{s^2+2s} \Big [ 
2\sum_{r=0}^{\infty}(-1)^rq^{r(r+1+4s)/2} - \sum_{r=1}^{4s}(-1)^rq^{r(r-1-4s)/2}\Big ]\\
& \ \ \ \ \ -2\sum_{r=0}^{\ell-1}\sum_{s=0}^{k-1}(-1)^{r}q^{r+3s+\binom{r}{2}+2rs+2\binom{s}{2}}\\
&= -\sum_{r=0}^{\ell-1}(-1)^rq^{\binom{r+1}{2}} \sum_{s=1}^{2r+1}q^{s(s-2-2r)}
 +2\sum_{s=0}^{k-1}q^{s^2+2s} \sum_{r=0}^{\infty}(-1)^rq^{r(r+1+4s)/2} \\
& \ \ \ \ \ -2\sum_{r=0}^{\ell-1}\sum_{s=0}^{k-1}(-1)^{r}q^{r+3s+\binom{r}{2}+2rs+2\binom{s}{2}},
\end{align*}}%
where for the last equality we used the easily shown fact that
\begin{equation}
\sum_{r=1}^{4s}(-1)^rq^{r(r-1-4s)/2}=0.
\end{equation}
One simply notes that
\begin{equation}
\sum_{r=1}^{4s}(-1)^rq^{r(r-1-4s)/2}=-\sum_{r=1}^{4s}(-1)^rq^{r(r-1-4s)/2}=0,\label{equation:revSum}
\end{equation}
where for the sum on the right we have made the substitution $r\rightarrow 4s+1-r$.\qedhere
\end{proof}
\begin{proof}[Proof of Identity (\ref{equation:Warnaar})] We use Proposition \ref{proposition:prop1} with $(\ell,k)=(-2,2)$.  We have
{\allowdisplaybreaks \begin{align*}
g_{1,2,2}(q,-q^3,q)&=qg_{1,2,2}(q^3,-q^3,q)
 -\sum_{r=0}^{-3}(-1)^rq^{\binom{r+1}{2}} \sum_{s=1}^{2r+1}q^{s(s-2-2r)}\\
& \ \ \ \ \ + 2\sum_{s=0}^{1}q^{s^2+2s} \sum_{r=0}^{\infty}(-1)^rq^{r(r+1+4s)/2} 
 -2\sum_{r=0}^{-3}\sum_{s=0}^{1}(-1)^{r}q^{r+3s+\binom{r}{2}+2rs+2\binom{s}{2}}\\
&=-g_{1,2,2}(q,-q^3,q)
 +\sum_{r=-2}^{-1}(-1)^rq^{\binom{r+1}{2}} \sum_{s=1}^{2r+1}q^{s(s-2-2r)}\\
& \ \ \ \ \ + 2\sum_{s=0}^{1}q^{s^2+2s}  
\sum_{r=0}^{\infty}(-1)^rq^{r(r+1+4s)/2} 
+2\sum_{r=-2}^{-1}\sum_{s=0}^{1}(-1)^{r}q^{r+3s+\binom{r}{2}+2rs+2\binom{s}{2}}\\
&=-g_{1,2,2}(q,-q^3,q)
+q \sum_{s=1}^{-3}q^{s(s+2)}
- \sum_{s=1}^{-1}q^{s^2}\\
& \ \ \ \ \ +  2\sum_{r=0}^{\infty}(-1)^rq^{r(r+1)/2}
 + 2q^{3} \sum_{r=0}^{\infty}(-1)^rq^{r(r+5)/2}  \\
& \ \ \ \ \  +2\sum_{r=-2}^{-1}\Big ( (-1)^{r}q^{r+\binom{r}{2}}
+(-1)^{r}q^{\binom{r}{2}+3r+3}\Big ) \\
&=-g_{1,2,2}(q,-q^3,q) -q \sum_{s=-2}^{0}q^{s(s+2)}
+ \sum_{s=0}^{0}q^{s^2}\\
& \ \ \ \ \ +  2\sum_{r=0}^{\infty}(-1)^rq^{r(r+1)/2}  + 2\sum_{r=0}^{\infty}(-1)^rq^{(r+2)(r+3)/2}
+ 2( q-1+1-q)\\
&=-g_{1,2,2}(q,-q^3,q)
 -2q +  2\sum_{r=0}^{\infty}(-1)^rq^{r(r+1)/2}  + 2\sum_{r=2}^{\infty}(-1)^rq^{r(r+1)/2}\\
&=-g_{1,2,2}(q,-q^3,q)
 -2 +  4\sum_{r=0}^{\infty}(-1)^rq^{r(r+1)/2},
\end{align*}}%
where in the second equality we used (\ref{equation:g-flip}).
\end{proof}

%%%%% 
%%%%%
%%%%%

\subsection{Proof of Identity (\ref{equation:Prop3pt4A})}
We begin with a proposition.
\begin{proposition}\label{proposition:prop2} We have
\begin{align}
g_{2,2,2}(-q^2,-q^3,q)
&=q^{2\ell}q^{3k}q^{2\binom{\ell}{2}+2\ell k+ 2\binom{k}{2}}g_{2,2,2}(-q^{2\ell+2k+2},-q^{2\ell + 2k+3},q)\label{equation:prop2eq}\\
& \ \ \ \ \ -\sum_{r=0}^{\ell-1}q^{r^2+r}\sum_{s=1}^{2r+1}q^{s(s-2-2r)}
-\sum_{s=0}^{k-1}q^{s^2+2s}\sum_{r=1}^{2s}q^{r(r-1-2s)}\nonumber\\
& \ \ \ \ \ -2\sum_{r=0}^{\ell-1}\sum_{s=0}^{k-1}q^{2r}q^{3s}q^{2\binom{r}{2}+2rs+2\binom{s}{2}}. \nonumber
\end{align}
\end{proposition}
\begin{proof}[Proof of Proposition \ref{proposition:prop2}]
We take the appropriate specializations of (\ref{equation:g-shift}) to have
{\allowdisplaybreaks \begin{align*}
g_{2,2,2}&(-q^2,-q^3,q)
-q^{2\ell}q^{3k}q^{2\binom{\ell}{2}+2\ell k+ 2\binom{k}{2}}g_{2,2,2}(-q^{2\ell+2k+2},-q^{2\ell + 2k+3},q)\\
&=\sum_{r=0}^{\ell-1}q^{r^2+r}\sum_{s\in\mathbb{Z}}\sg(s)q^{s^2+2s+2rs} 
+\sum_{s=0}^{k-1}q^{s^2+2s}\sum_{r\in\mathbb{Z}}^{\infty}\sg(r)q^{r^2+r+2rs}\\
& \ \ \ \ \ -2\sum_{r=0}^{\ell-1}\sum_{s=0}^{k-1}q^{2r}q^{3s}q^{2\binom{r}{2}+2rs+2\binom{s}{2}}\\
&=\sum_{r=0}^{\ell-1}q^{r^2+r}(\sum_{s=0}^{\infty}q^{s(s+2+2r)} 
-\sum_{s=1}^{\infty}q^{s(s-2-2r)} ) \\
& \ \ \ \ \ +\sum_{s=0}^{k-1}q^{s^2+2s}(\sum_{r=0}^{\infty}q^{r(r+1+2s)}
-\sum_{r=1}^{\infty}q^{r(r-1-2s)})\\
& \ \ \ \ \ -2\sum_{r=0}^{\ell-1}\sum_{s=0}^{k-1}q^{2r}q^{3s}q^{2\binom{r}{2}+2rs+2\binom{s}{2}}\\
&=-\sum_{r=0}^{\ell-1}q^{r^2+r}\sum_{s=1}^{2r+1}q^{s(s-2-2r)}
 -\sum_{s=0}^{k-1}q^{s^2+2s}\sum_{r=1}^{2s}q^{r(r-1-2s)}\\
& \ \ \ \ \ -2\sum_{r=0}^{\ell-1}\sum_{s=0}^{k-1}q^{2r}q^{3s}q^{2\binom{r}{2}+2rs+2\binom{s}{2}}.\qedhere
\end{align*}}%
\end{proof}

\begin{proof}[Proof of Identity (\ref{equation:Prop3pt4A})]
We consider two cases of the general formula (\ref{equation:prop2eq}).  The first when $(\ell,k)=(1,0)$ and the second when $(\ell,k)=(0,1)$.  We have
{\allowdisplaybreaks \begin{align}
g_{2,2,2}(-q^2,-q^3,q)&=q^2g_{2,2,2}(-q^4,-q^5,q)-q^{-1}\label{equation:Prop3pt4A-pre1},\\
g_{2,2,2}(-q^2,-q^3,q)&=q^3g_{2,2,2}(-q^4,-q^5,q)\label{equation:Prop3pt4A-pre2}.
\end{align}}%
The result follows from multiplying (\ref{equation:Prop3pt4A-pre1}) by $q$ and subtracting the result from (\ref{equation:Prop3pt4A-pre2}).
\end{proof}

%%%%% 
%%%%%
%%%%%

\subsection{Proof of Identity (\ref{equation:Prop3pt4B})}
We have a proposition.
\begin{proposition}\label{proposition:prop3} We have
\begin{align}
g_{2,2,2}(q^2,-q^3,q)
& = (-q^2)^{\ell}q^{3k}q^{2\binom{\ell}{2}+2\ell k+ 2\binom{k}{2}}g_{2,2,2}(q^{2\ell+2k+2},-q^{2\ell + 2k+3},q) \label{equation:prop3eq} \\
& \ \ \ \ \ -\sum_{r=0}^{\ell-1}(-1)^{r}q^{r^2+r}\sum_{s=1}^{2r+1}q^{s(s-2-2r)}
+2\sum_{s=0}^{k-1}q^{s^2+2s}\sum_{r=0}^{\infty}(-1)^rq^{r(r+1+2s)} \nonumber\\
& \ \ \ \ \ -2\sum_{r=0}^{\ell-1}\sum_{s=0}^{k-1}(-1)^rq^{2r}q^{3s}q^{2\binom{r}{2}+2rs+2\binom{s}{2}}. \nonumber
\end{align}
\end{proposition}
\begin{proof}[Proof of Proposition \ref{proposition:prop3}]
We take the appropriate specialization of (\ref{equation:g-shift}) to have
{\allowdisplaybreaks \begin{align*}
g_{2,2,2}&(q^2,-q^3,q)
 - (-q^2)^{\ell}q^{3k}q^{2\binom{\ell}{2}+2\ell k+ 2\binom{k}{2}}g_{2,2,2}(q^{2\ell+2k+2},-q^{2\ell + 2k+3},q) \\
&=\sum_{r=0}^{\ell-1}(-1)^{r}q^{r^2+r}\sum_{s\in\mathbb{Z}}\sg(s)q^{s^2+2s+2rs}
  +\sum_{s=0}^{k-1}q^{s^2+2s}\sum_{r=0}^{\infty}\sg(r)(-1)^rq^{r^2+r+2rs}\\
& \ \ \ \ \ -2\sum_{r=0}^{\ell-1}\sum_{s=0}^{k-1}(-1)^rq^{2r}q^{3s}q^{2\binom{r}{2}+2rs+2\binom{s}{2}} \\
&=\sum_{r=0}^{\ell-1}(-1)^{r}q^{r^2+r}(\sum_{s=0}^{\infty}q^{s(s+2+2r)} 
-\sum_{s=1}^{\infty}q^{s(s-2-2r)} )\\
& \ \ \ \ \   +\sum_{s=0}^{k-1}q^{s^2+2s}(\sum_{r=0}^{\infty}(-1)^rq^{r(r+1+2s)}
-\sum_{r=1}^{\infty}(-1)^rq^{r(r-1-2s)})\\
& \ \ \ \ \ -2\sum_{r=0}^{\ell-1}\sum_{s=0}^{k-1}(-1)^rq^{2r}q^{3s}q^{2\binom{r}{2}+2rs+2\binom{s}{2}} \\
&=-\sum_{r=0}^{\ell-1}(-1)^{r}q^{r^2+r}\sum_{s=1}^{2r+1}q^{s(s-2-2r)}\\
& \ \ \ \ \   +\sum_{s=0}^{k-1}q^{s^2+2s}(2\sum_{r=0}^{\infty}(-1)^rq^{r(r+1+2s)}-\sum_{r=1}^{2s}(-1)^rq^{r(r-1-2s)})\\
& \ \ \ \ \ -2\sum_{r=0}^{\ell-1}\sum_{s=0}^{k-1}(-1)^rq^{2r}q^{3s}q^{2\binom{r}{2}+2rs+2\binom{s}{2}} \\
&=-\sum_{r=0}^{\ell-1}(-1)^{r}q^{r^2+r}\sum_{s=1}^{2r+1}q^{s(s-2-2r)}
  +2\sum_{s=0}^{k-1}q^{s^2+2s}\sum_{r=0}^{\infty}(-1)^rq^{r(r+1+2s)}\\
& \ \ \ \ \ -2\sum_{r=0}^{\ell-1}\sum_{s=0}^{k-1}(-1)^rq^{2r}q^{3s}q^{2\binom{r}{2}+2rs+2\binom{s}{2}},
\end{align*}}%
where for the last equality we argued as in (\ref{equation:revSum}).
\end{proof}
\begin{proof}[Proof of Identity (\ref{equation:Prop3pt4B})]  We consider two cases of the general formula (\ref{equation:prop3eq}).  The first when $(\ell,k)=(1,0)$ and the second when $(\ell,k)=(0,1)$.  We have
\begin{align}
g_{2,2,2}(q^2,-q^3,q)&=-q^2g_{2,2,2}(q^4,-q^5,q)-q^{-1},\label{equation:Prop3pt4B-pre1}\\
g_{2,2,2}(q^2,-q^3,q)&=q^3g_{2,2,2}(q^4,-q^5,q)+2\sum_{r=0}^{\infty}(-1)^rq^{r(r+1)}
\label{equation:Prop3pt4B-pre2}.
\end{align}
Multiplying (\ref{equation:Prop3pt4B-pre1}) by $q$ and adding the result to (\ref{equation:Prop3pt4B-pre2}) yields the result.
\end{proof}

%%%%% 
%%%%%
%%%%%

\subsection{Proof of Identity (\ref{equation:Thm1pt1A})} We begin with a proposition:
\begin{proposition}\label{proposition:prop4}
\begin{align}
g_{1,2,4}(-q,-q^4,q)&=q^{\ell}q^{4k}q^{\binom{\ell}{2}+2\ell k+ 4\binom{k}{2}}g_{1,2,4}(-q^{\ell+2k+1},-q^{2\ell + 4k+4},q)\label{equation:prop4eq} \\
& \ \ \ \ \ -\sum_{r=0}^{\ell-1}q^{\binom{r+1}{2}}\sum_{s=1}^{r}q^{2s(s-1-r)}
-\sum_{s=0}^{k-1}q^{2s^2+2s}\sum_{r=1}^{4s}q^{r(r-1-4s)/2}\nonumber \\
& \ \ \ \ \ -2\sum_{r=0}^{\ell-1}\sum_{s=0}^{k-1}q^{r+4s}q^{\binom{r}{2}+2rs+4\binom{s}{2}}. \nonumber
\end{align}
\end{proposition}
\begin{proof}[Proof of Proposition \ref{proposition:prop4}]
We take the appropriate specialization of (\ref{equation:g-shift}) to have
{\allowdisplaybreaks \begin{align*}
g_{1,2,4}&(-q,-q^4,q)-q^{\ell}q^{4k}q^{\binom{\ell}{2}+2\ell k+ 4\binom{k}{2}}g_{1,2,4}(-q^{\ell+2k+1},-q^{2\ell + 4k+4},q)\\
&=\sum_{r=0}^{\ell-1}q^{\binom{r+1}{2}}\sum_{s\in\mathbb{Z}}\sg(s)q^{2s^2+2s+2rs}
 +\sum_{s=0}^{k-1}q^{2s^2+2s}\sum_{r\in\mathbb{Z}}\sg(r)q^{r(r+1+4s)/2}\\
&\ \ \ \ \ -2\sum_{r=0}^{\ell-1}\sum_{s=0}^{k-1}q^{r+4s}q^{\binom{r}{2}+2rs+4\binom{s}{2}}\\
&=\sum_{r=0}^{\ell-1}q^{\binom{r+1}{2}}( \sum_{s=0}^{\infty}q^{2s(s+1+r)}
-\sum_{s=1}^{\infty}q^{2s(s-1-r)})\\
& \ \ \ \ \  +\sum_{s=0}^{k-1}q^{2s^2+2s}(\sum_{r=0}^{\infty}q^{r(r+1+4s)/2}
-\sum_{r=1}^{\infty}q^{r(r-1-4s)/2})\\
&\ \ \ \ \ -2\sum_{r=0}^{\ell-1}\sum_{s=0}^{k-1}q^{r+4s}q^{\binom{r}{2}+2rs+4\binom{s}{2}}\\
&= -\sum_{r=0}^{\ell-1}q^{\binom{r+1}{2}}\sum_{s=1}^{r}q^{2s(s-1-r)}
-\sum_{s=0}^{k-1}q^{2s^2+2s}\sum_{r=1}^{4s}q^{r(r-1-4s)/2}\nonumber \\
& \ \ \ \ \ -2\sum_{r=0}^{\ell-1}\sum_{s=0}^{k-1}q^{r+4s}q^{\binom{r}{2}+2rs+4\binom{s}{2}}.\qedhere
\end{align*}}%
\end{proof}
\begin{proof}[Proof of Identity ( \ref{equation:Thm1pt1A})]  The specializations $(\ell,k)=(0,1)$ and $(\ell,k)=(2,0)$ of (\ref{equation:prop4eq}) read
{\allowdisplaybreaks \begin{align}
g_{1,2,4}(-q,-q^4,q)&=q^4g_{1,2,4}(-q^3,-q^8,q)\label{equation:Thm1pt1A-pre1},\\
g_{1,2,4}(-q,-q^4,q)&=q^3g_{1,2,4}(-q^3,-q^8,q)-q^{-1}\label{equation:Thm1pt1A-pre2}.
\end{align}}%
We multiply (\ref{equation:Thm1pt1A-pre2}) by $q$ and subtract the result from ( \ref{equation:Thm1pt1A-pre1}) to obtain the result.
\end{proof}

%%%%% 
%%%%%
%%%%%

\subsection{Proof of Identity (\ref{equation:Thm1pt1B})} We start with a proposition:
\begin{proposition}\label{proposition:prop5}
\begin{align}
g_{1,2,4}(q,-q^4,q)& = 
(-q)^{\ell}q^{4k}q^{\binom{\ell}{2}+2\ell k+ 4\binom{k}{2}}g_{1,2,4}(q^{\ell+2k+1},-q^{2\ell + 4k+4},q)\label{equation:prop5eq}\\
& \ \ \ \ \ -\sum_{r=0}^{\ell-1}(-1)^{r}q^{\binom{r+1}{2}}\sum_{s=1}^{r}q^{2s(s-1-r)}
 +2\sum_{s=0}^{k-1}q^{2s^2+2s} \sum_{r=0}^{\infty}(-1)^rq^{r(r+1+4s)/2}
\nonumber\\
& \ \ \ \ \ -2\sum_{r=0}^{\ell-1}\sum_{s=0}^{k-1}(-1)^{r}q^{r}q^{4s}q^{\binom{r}{2}+2rs+4\binom{s}{2}}. \nonumber
\end{align}
\end{proposition}
\begin{proof}[Proof of Proposition \ref{proposition:prop5}]
We take the appropriate specialization of (\ref{equation:g-shift}) to have
{\allowdisplaybreaks \begin{align*}
g_{1,2,4}&(q,-q^4,q) - 
(-q)^{\ell}q^{4k}q^{\binom{\ell}{2}+2\ell k+ 4\binom{k}{2}}g_{1,2,4}(q^{\ell+2k+1},-q^{2\ell + 4k+4},q)\\
&=\sum_{r=0}^{\ell-1}(-1)^{r}q^{\binom{r+1}{2}}\Big [ \sum_{s=0}^{\infty}q^{2s(s+1+r)}
-\sum_{s=1}^{\infty}q^{2s(s-1-r)}\Big ] \\
& \ \ \ \ \ +\sum_{s=0}^{k-1}q^{2s^2+2s}\Big [ \sum_{r=0}^{\infty}(-1)^rq^{r(r+1+4s)/2}
-\sum_{r=1}^{\infty}(-1)^rq^{r(r-1-4s)/2}\Big ]\\
& \ \ \ \ \ -2\sum_{r=0}^{\ell-1}\sum_{s=0}^{k-1}(-1)^{r}q^{r}q^{4s}q^{\binom{r}{2}+2rs+4\binom{s}{2}}\\
&=-\sum_{r=0}^{\ell-1}(-1)^{r}q^{\binom{r+1}{2}}\sum_{s=1}^{r}q^{2s(s-1-r)}\\
& \ \ \ \ \ +\sum_{s=0}^{k-1}q^{2s^2+2s}\Big [ 2\sum_{r=0}^{\infty}(-1)^rq^{r(r+1+4s)/2}
-\sum_{r=1}^{4s}(-1)^rq^{r(r-1-4s)/2}\Big ]\\
& \ \ \ \ \ -2\sum_{r=0}^{\ell-1}\sum_{s=0}^{k-1}(-1)^{r}q^{r}q^{4s}q^{\binom{r}{2}+2rs+4\binom{s}{2}}\\
&=-\sum_{r=0}^{\ell-1}(-1)^{r}q^{\binom{r+1}{2}}\sum_{s=1}^{r}q^{2s(s-1-r)}
+2\sum_{s=0}^{k-1}q^{2s^2+2s}\sum_{r=0}^{\infty}(-1)^rq^{r(r+1+4s)/2}\\
& \ \ \ \ \ -2\sum_{r=0}^{\ell-1}\sum_{s=0}^{k-1}(-1)^{r}q^{r}q^{4s}q^{\binom{r}{2}+2rs+4\binom{s}{2}},
\end{align*}}%
where for the last equality we argued as in (\ref{equation:revSum}).
\end{proof}
\begin{proof}[Proof of Identity (\ref{equation:Thm1pt1B})]  The specializations $(\ell,k)=(0,1)$ and $(\ell,k)=(2,0)$ of (\ref{equation:prop5eq}) read:
{\allowdisplaybreaks \begin{align}
g_{1,2,4}(q,-q^4,q)&=q^4g_{1,2,4}(q^3,-q^8,q)+2\sum_{r=0}^{\infty}(-1)^rq^{r(r+1)/2},\label{equation:Thm1pt1B-pre1}\\
g_{1,2,4}(q,-q^4,q)&=q^3g_{1,2,4}(q^3,-q^8,q)+q^{-1}.\label{equation:Thm1pt1B-pre2}
\end{align}}%
The result follows from multiplying (\ref{equation:Thm1pt1B-pre2}) by $q$ and subtracting the result from (\ref{equation:Thm1pt1B-pre1}).
\end{proof}

%%%%% 
%%%%%
%%%%%

\subsection{Proof of Identity (\ref{equation:Thm3pt6})} We begin with a proposition.
\begin{proposition}\label{proposition:prop6}  We have
\begin{align}
g_{6,3,2}(-q^5,-q^3,q) & = q^{5\ell}q^{3k}q^{6\binom{\ell}{2}+3\ell k+ 2\binom{k}{2}}g_{6,3,2}(-q^{6\ell+3k+5},-q^{3\ell + 2k+3},q)\label{equation:prop6eqA}\\
& \ \ \ \ \ -\sum_{r=0}^{\ell-1}q^{3r^3+2r} \sum_{s=1}^{3r+1}q^{s(s-2-3r)}\nonumber \\
& \ \ \ \ \ +\sum_{s=0}^{k-1}q^{s^2+2s}\Big [ \sum_{r=0}^{\infty}q^{r(3r+2+3s)}
-\sum_{r=1}^{\infty}q^{r(3r-2-3s)}\Big ]\nonumber  \\
& \ \ \ \ \ -2\sum_{r=0}^{\ell-1}\sum_{s=0}^{k-1}q^{5r}q^{3s}q^{6\binom{r}{2}+3rs+2\binom{s}{2}}. \nonumber
\end{align}
\end{proposition}

\begin{proof}[Proof of Proposition \ref{proposition:prop6}]
We take the appropriate specialization of (\ref{equation:g-shift}) to have
{\allowdisplaybreaks \begin{align*}
g_{6,3,2}&(-q^5,-q^3,q) - q^{5\ell}q^{3k}q^{6\binom{\ell}{2}+3\ell k+ 2\binom{k}{2}}g_{6,3,2}(-q^{6\ell+3k+5},-q^{3\ell + 2k+3},q)\\
&=\sum_{r=0}^{\ell-1}q^{3r^3+2r}\Big [ \sum_{s=0}^{\infty}q^{s(s+2+3r)}
- \sum_{s=1}^{\infty}q^{s(s-2-3r)}\Big ]\\
&\ \ \ \ \ +\sum_{s=0}^{k-1}q^{s^2+2s}\Big [ \sum_{r=0}^{\infty}q^{r(3r+2+3s)}
-\sum_{r=1}^{\infty}q^{r(3r-2-3s)}\Big ]\\
& \ \ \ \ \ -2\sum_{r=0}^{\ell-1}\sum_{s=0}^{k-1}q^{5r}q^{3s}q^{6\binom{r}{2}+3rs+2\binom{s}{2}}\\
&=-\sum_{r=0}^{\ell-1}q^{3r^3+2r} \sum_{s=1}^{3r+1}q^{s(s-2-3r)} 
 +\sum_{s=0}^{k-1}q^{s^2+2s}\Big [ \sum_{r=0}^{\infty}q^{r(3r+2+3s)}
-\sum_{r=1}^{\infty}q^{r(3r-2-3s)}\Big ]  \\
& \ \ \ \ \ -2\sum_{r=0}^{\ell-1}\sum_{s=0}^{k-1}q^{5r}q^{3s}q^{6\binom{r}{2}+3rs+2\binom{s}{2}}.\qedhere
\end{align*}}%
\end{proof}
Instead of proving (\ref{equation:Thm3pt6}), we will prove an equivalent form found in the following lemma.
\begin{lemma}\label{lemma:Thm3pt6}  We have that 
\begin{equation}
G(q)=q^{-1}-\sum_{n=0}^{\infty}q^{3n^2+2n-1}+\sum_{n=0}^{\infty}q^{3n^2+4n}
 -2\sum_{n=0}^{\infty}q^{3n^2+5n+1}(1-q^{2n+2}).\label{equation:newThm3pt6}
\end{equation}
\end{lemma}
\begin{proof}[Proof of Lemma \ref{lemma:Thm3pt6}]
We begin with
\begin{equation*}
q^4G(q^3)=q-\sum_{n=1}^{\infty}\Big ( \frac{n}{3}\Big )q^{n^2}-2\sum_{n=0}^{\infty}q^{9n^2+15n+7}(1-q^{6n+6}). 
\end{equation*}
We divide out by $q^4$ to have
\begin{equation*}
G(q^3)=q^{-3}-\sum_{n=1}^{\infty}\Big ( \frac{n}{3}\Big )q^{n^2-4}-2\sum_{n=0}^{\infty}q^{9n^2+15n+3}(1-q^{6n+6}). 
\end{equation*}
We expand the Legendre symbol to have
\begin{align*}
G(q^3)&=q^{-3}-\sum_{n=0}^{\infty}q^{(3n+1)^2-4}+\sum_{n=0}^{\infty}q^{(3n+2)^2-4}
 -2\sum_{n=0}^{\infty}q^{9n^2+15n+3}(1-q^{6n+6}),\\
 &=q^{-3}-\sum_{n=0}^{\infty}q^{9n^2+6n-3}+\sum_{n=0}^{\infty}q^{9n^2+12n}
 -2\sum_{n=0}^{\infty}q^{9n^2+15n+3}(1-q^{6n+6}),
\end{align*}
and then make the substitution $q\rightarrow q^{1/3}$.
\end{proof}

\begin{proof}[Proof of Identity (\ref{equation:newThm3pt6})] We apply Proposition \ref{proposition:prop6}  with $(\ell,k)=(2,-3)$ to the first term and apply (\ref{equation:g-flip}) to the second term to have
{\allowdisplaybreaks \begin{align*}
g_{6,3,2}&(-q^5,-q^3,q) - qg_{6,3,2}(-q^{7},-q^{4},q)\\
& =qg_{6,3,2}(-q^{8},-q^{3},q)- qg_{6,3,2}(-q^{8},-q^{3},q)
 -\sum_{r=0}^{1}q^{3r^3+2r} \sum_{s=1}^{3r+1}q^{s(s-2-3r)} \\
& \ \ \ \ \ +\sum_{s=0}^{-4}q^{s^2+2s}\Big [ \sum_{r=0}^{\infty}q^{r(3r+2+3s)}
-\sum_{r=1}^{\infty}q^{r(3r-2-3s)}\Big ]
 -2\sum_{r=0}^{1}\sum_{s=0}^{-4}q^{5r}q^{3s}q^{6\binom{r}{2}+3rs+2\binom{s}{2}}\\
%%%
& = 
  -\sum_{s=1}^{1}q^{s(s-2)} 
 -q^{5} \sum_{s=1}^{4}q^{s(s-5)}\\
& \ \ \ \ \ -\sum_{s=-3}^{-1}q^{s^2+2s}\Big [ \sum_{r=0}^{\infty}q^{r(3r+2+3s)}
-\sum_{r=1}^{\infty}q^{r(3r-2-3s)}\Big ]
 +2\sum_{r=0}^{1}\sum_{s=-3}^{-1}q^{5r}q^{3s}q^{6\binom{r}{2}+3rs+2\binom{s}{2}}\\
& =
  2q^3+2+3q^{-1}- \sum_{r=0}^{\infty}q^{r(3r-7)+3}
+	\sum_{r=1}^{\infty}q^{r(3r+7)+3} \\
& \ \ \ \ \ - \sum_{r=0}^{\infty}q^{r(3r-4)}
+\sum_{r=1}^{\infty}q^{r(3r+4)}
 -\sum_{r=0}^{\infty}q^{r(3r-1)-1}
+\sum_{r=1}^{\infty}q^{r(3r+1)-1}\\
  %%%
& = 
    q^{-1}-\sum_{r=2}^{\infty}q^{r(3r-7)+3}
+	\sum_{r=0}^{\infty}q^{r(3r+7)+3}\\
& \ \ \ \ \ -\sum_{r=1}^{\infty}q^{r(3r-4)}
+\sum_{r=0}^{\infty}q^{r(3r+4)}
 -\sum_{r=1}^{\infty}q^{r(3r-1)-1}
+\sum_{r=1}^{\infty}q^{r(3r+1)-1}\\
%%%
& =  q^{-1}
 -\sum_{r=0}^{\infty}q^{3r^2+2r-1}+\sum_{r=0}^{\infty}q^{r(3r+4)}
-2\sum_{r=0}^{\infty}q^{3r^2+5r+1}(1-q^{2r+2}).\qedhere
\end{align*} }%
\end{proof}

\section{False theta function identities of Andrews and Warnaar}\label{section:andrewsWarnaar}
In addition to the false theta function identities mentioned in the Introduction, Andrews and Warnaar showed
\cite[$(1.5)$, Theorem $13$, $(1.7)$]{AW}
{\allowdisplaybreaks \begin{align}
\sum_{n=0}^{\infty}\frac{q^n}{(-q;q^2)_{n+1}}
&=\sum_{n=0}^{\infty}(-1)^{n}q^{2n(3n+2)}(1+q^{4n+2}),\\
\sum_{n,j=0}^{\infty}\frac{q^{j(2j+1)+n}(q^2;q^2)_{n+j}}{(-q;q)_{2n+2j+1}(q^2;q^2)_j(q^2;q^2)_{n}}
&=\sum_{n=0}(-1)^{n}q^{n(5n+3)}(1+q^{4n+2}),\\
\sum_{n=0}^{\infty}\frac{(-q^4;q^4)_{n-1}q^n}{(-q^2;q^2)_{n-1}(-q;q^2)_{n+1}}
&=\frac{1}{1+q^2}\Big ( 1+2q^2\sum_{n=0}^{\infty}(-1)^nq^{4n(n+1)}\Big ), 
\end{align}}%
where a typo has been corrected.  Intermediate steps of their proofs include respectively
{\allowdisplaybreaks \begin{align}
\sum_{n=0}^{\infty}\frac{q^n}{(-q;q^2)_{n+1}}&
=g_{3,1,3}(q^6,q^{10},q^4)+q^2g_{3,1,3}(q^{10},q^{14},q^4),\\
\sum_{n,j=0}^{\infty}\frac{q^{j(2j+1)+n}(q^2;q^2)_{n+j}}{(-q;q)_{2n+2j+1}(q^2;q^2)_j(q^2;q^2)_{n}}
&=g_{5,3,5}(q^6,q^{8},q^2)+q^2g_{5,3,5}(q^{10},q^{12},q^2),\\
\sum_{n=0}^{\infty}\frac{(-q^4;q^4)_{n-1}q^n}{(-q^2;q^2)_{n-1}(-q;q^2)_{n+1}}
&=g_{1,1,1}(q^{4},q^{8},q^{8})+q^{2}g_{1,1,1}(q^{8},q^{12},q^{8}).
\end{align}}%

With Andrews and Warnaar's identities from the Introduction in mind, we will give new proofs that
{\allowdisplaybreaks \begin{align}
g_{1,3,1}(q,q^2,q)&=\sum_{n=0}^{\infty}(-1)^nq^{n(n+1)/2}\label{equation:AW-1pt1a-g-form},\\
g_{3,1,3}(q^2,q^3,q)&=\sum_{n=0}^{\infty}(-1)^nq^{3n(n+1)/2}\label{equation:AW-1pt1c-g-form},\\
g_{1,0,1}(q^2,q^4,q^4)&=\sum_{n=0}^{\infty}(-1)^nq^{2n(n+1)}\label{equation:AW-1pt1d-g-form}.
\end{align}}%
With the identities mentioned in this section in mind, we will give new proofs that
{\allowdisplaybreaks \begin{align}
g_{3,1,3}(q^6,q^{10},q^4)+q^2g_{3,1,3}(q^{10},q^{14},q^4)
&=\sum_{n=0}^{\infty}(-1)^{n}q^{2n(3n+2)}(1+q^{4n+2}),\label{equation:AW-1pt5-g-form}\\
g_{5,3,5}(q^6,q^{8},q^2)+q^2g_{5,3,5}(q^{10},q^{12},q^2)
&=\sum_{n=0}^{\infty}(-1)^{n}q^{n(5n+3)}(1+q^{4n+2}),\label{equation:AW-Thm13-g-form}\\
g_{1,1,1}(q^{4},q^{8},q^{8})+q^{2}g_{1,1,1}(q^{12},q^{8},q^{8})
&=\frac{1}{1+q^2}\Big ( 1+2q^2\sum_{n=0}^{\infty}(-1)^nq^{4n(n+1)}\Big ).\label{equation:AW-Thm14-g-form}
\end{align}}%

\begin{proof}[Proof of Identity (\ref{equation:AW-1pt1a-g-form})]
Let us take (\ref{equation:g-shift}) with $(\ell,k)=(0,1)$.  We have
{\allowdisplaybreaks \begin{align*}
g_{1,3,1}(q,q^2,q)&=-q^2g_{1,3,1}(q^4,q^3,q)+\sum_{r\in\mathbb{Z}}\sg(r)(-q)^{r}q^{\binom{r}{2}}\\
&=-q^2g_{1,3,1}(q^4,q^3,q)+\sum_{r=0}^{\infty}(-1)^rq^{\binom{r+1}{2}}
-\sum_{r=1}^{\infty}(-1)^rq^{\binom{r}{2}}\\
&=-q^2g_{1,3,1}(q^4,q^3,q)+2\sum_{r=0}^{\infty}(-1)^rq^{\binom{r+1}{2}}\\
&=-g_{1,3,1}(q,q^2,q)+2\sum_{r=0}^{\infty}(-1)^rq^{\binom{r+1}{2}},
\end{align*}}%
where we have used (\ref{equation:g-flip}).  Comparing the extremes of the above derivation, we have
\begin{equation*}
g_{1,3,1}(q,q^2,q)=-g_{1,3,1}(q,q^2,q)+2\sum_{r=0}^{\infty}(-q)^{r}q^{\binom{r}{2}},
\end{equation*}
and the result follows.
\end{proof}

\begin{proof}[Proof of Identity (\ref{equation:AW-1pt1c-g-form})]
Let us take (\ref{equation:g-shift}) with $(\ell,k)=(1,0)$.  We have
{\allowdisplaybreaks \begin{align*}
g_{3,1,3}(q^2,q^3,q)&=-q^2g_{3,1,3}(q^5,q^4,q)
+\sum_{s\in\mathbb{Z}}\sg(s)(-1)^{s}q^{3s}q^{3\binom{s}{2}}\\
&=-q^2g_{3,1,3}(q^5,q^4,q)
+\sum_{s=0}^{\infty}(-1)^{s}q^{3\binom{s+1}{2}}
-\sum_{s=1}^{\infty}(-1)^{s}q^{3\binom{s}{2}}\\
&=-q^2g_{3,1,3}(q^5,q^4,q)
+2\sum_{s=0}^{\infty}(-1)^{s}q^{3\binom{s+1}{2}}\\
&=-g_{3,1,3}(q^2,q^3,q)
+2\sum_{s=0}^{\infty}(-1)^{s}q^{3\binom{s+1}{2}},
\end{align*}}%
where we have used (\ref{equation:g-flip}).  The result follows.
\end{proof}

\begin{proof}[Proof of Identity (\ref{equation:AW-1pt1d-g-form})]
Let us take (\ref{equation:g-shift}) with $(\ell,k)=(1,0)$.  We have
{\allowdisplaybreaks \begin{align*}
g_{1,0,1}(q^2,q^4,q^4)&=-q^2g_{1,0,1}(q^6,q^4,q^4)
+\sum_{s\in\mathbb{Z}}\sg(s)(-1)^{s}q^{4s}q^{4\binom{s}{2}}\\
&=-q^2g_{1,0,1}(q^6,q^4,q^4)
+\sum_{s=0}^{\infty}(-1)^{s}q^{2s(s+1)}
-\sum_{s=1}^{\infty}(-1)^{s}q^{2s(s-1)}\\
&=-q^2g_{1,0,1}(q^6,q^4,q^4)
+2\sum_{s=0}^{\infty}(-1)^{s}q^{2s(s+1)}\\
&=-g_{1,0,1}(q^2,q^4,q^4)
+2\sum_{s=0}^{\infty}(-1)^{s}q^{2s(s+1)},
\end{align*}}%
where we have used (\ref{equation:g-flip}).  The result follows.
\end{proof}

\begin{proof}[Proof of Identity (\ref{equation:AW-1pt5-g-form})]
Let us take (\ref{equation:g-shift}) with $(\ell,k)=(1,0)$.  We have
{\allowdisplaybreaks \begin{align*}
g_{3,1,3}(q^{6},q^{10},q^{4})
&=-q^6g_{3,1,3}(q^{18},q^{14},q^{4})+\sum_{s\in\mathbb{Z}}\sg(s)(-1)^{s}q^{10s}q^{12\binom{s}{2}}\\
&=-q^6g_{3,1,3}(q^{18},q^{14},q^{4})+\sum_{s=0}^{\infty}(-1)^{s}q^{6s^2+4s}
-\sum_{s=1}^{\infty}(-1)^{s}q^{6s^2-4s}\\
&=-q^6g_{3,1,3}(q^{18},q^{14},q^{4})+\sum_{s=0}^{\infty}(-1)^{s}q^{6s^2+4s}(1+q^{4s+2})\\
&=-q^{2}g_{3,1,3}(q^{10},q^{14},q^{4})+\sum_{s=0}^{\infty}(-1)^{s}q^{6s^2+4s}(1+q^{4s+2}),
\end{align*}}%
where we have used (\ref{equation:g-flip}).  The result follows. 
\end{proof}

\begin{proof}[Proof of Identity (\ref{equation:AW-Thm13-g-form})]
Let us take (\ref{equation:g-shift}) with $(\ell,k)=(1,0)$.  We have
{\allowdisplaybreaks \begin{align*}
g_{5,3,5}(q^{6},q^{8},q^{2})&=-q^{6}g_{5,3,5}(q^{16},q^{14},q^{2})
+\sum_{s\in\mathbb{Z}}\sg(s)(-1)^{s}q^{8s}q^{10\binom{s}{2}}\\
&=-q^{6}g_{5,3,5}(q^{16},q^{14},q^{2})
+\sum_{s=0}^{\infty}(-1)^{s}q^{5s^2+3s}-\sum_{s=1}^{\infty}(-1)^{s}q^{5s^2-3s}\\
&=-q^{6}g_{5,3,5}(q^{16},q^{14},q^{2})
+\sum_{s=0}^{\infty}(-1)^{s}q^{5s^2+3s}(1+q^{4s+2})\\
&=-q^{2}g_{5,3,5}(q^{10},q^{12},q^{2})
+\sum_{s=0}^{\infty}(-1)^{s}q^{5s^2+3s}(1+q^{4s+2}),
\end{align*}}%
where we have used (\ref{equation:g-flip}).  The result follows.
\end{proof}

\begin{proof}[Proof of Identity (\ref{equation:AW-Thm14-g-form})]
Let us take (\ref{equation:g-shift}) with $(\ell,k)=(1,0)$.  We have
{\allowdisplaybreaks \begin{align*}
g_{1,1,1}&(q^{4},q^{8},q^{8})+q^{2}g_{1,1,1}(q^{8},q^{12},q^{8})\\
&=-q^4g_{1,1,1}(q^{12},q^{16},q^{8})-q^{10}g_{1,1,1}(q^{16},q^{20},q^{8})\\
& \ \ \ \ \ +\sum_{s\in\mathbb{Z}}\sg(s)(-1)^{s}q^{8s}q^{8\binom{s}{2}}
+q^2\sum_{s\in\mathbb{Z}}\sg(s)(-1)^{s}q^{12s}q^{8\binom{s}{2}}\\
&=-q^4g_{1,1,1}(q^{12},q^{16},q^{8})-q^{10}g_{1,1,1}(q^{16},q^{20},q^{8})\\
& \ \ \ \ \ +\sum_{s=0}^{\infty}(-1)^{s}q^{4s^2+4s} -\sum_{s=1}^{\infty}(-1)^{s}q^{4s^2-4s}\\
& \ \ \ \ \ +\sum_{s=0}^{\infty}(-1)^{s}q^{4s^2+8s+2} -\sum_{s=1}^{\infty}(-1)^{s}q^{4s^2-8s+2}\\
&=-q^4g_{1,1,1}(q^{12},q^{16},q^{8})-q^{10}g_{1,1,1}(q^{16},q^{20},q^{8})\\
& \ \ \ \ \ +2\sum_{s=0}^{\infty}(-1)^{s}q^{4s^2+4s} 
 +q^2\sum_{s=0}^{\infty}(-1)^{s}q^{4s(s+2)} -q^2\sum_{s=1}^{\infty}(-1)^{s}q^{4s(s-2)}\\
&=-q^4g_{1,1,1}(q^{12},q^{16},q^{8})-q^{10}g_{1,1,1}(q^{16},q^{20},q^{8})
+2\sum_{s=0}^{\infty}(-1)^{s}q^{4s^2+4s}+q^{-2}\\
&=-g_{1,1,1}(q^{12},q^{8},q^{8})-q^{-2}g_{1,1,1}(q^{8},q^{4},q^{8})
+2\sum_{s=0}^{\infty}(-1)^{s}q^{4s^2+4s}+q^{-2},
\end{align*}}%
where we have used (\ref{equation:g-flip}).  Comparing the extremes of the above derivation yields
{\allowdisplaybreaks \begin{align*}
g_{1,1,1}&(q^{4},q^{8},q^{8})+q^{2}g_{1,1,1}(q^{12},q^{8},q^{8})\\
&=-g_{1,1,1}(q^{12},q^{8},q^{8})-q^{-2}g_{1,1,1}(q^{4},q^{8},q^{8})
+2\sum_{s=0}^{\infty}(-1)^{s}q^{4s^2+4s}+q^{-2}
\end{align*}}%
or equivalently
\begin{align*}
 g_{1,1,1}(q^{4},q^{8},q^{8})+q^{2}g_{1,1,1}(q^{12},q^{8},q^{8})
&= \frac{q^2}{1+q^2}\Big ( 2\sum_{s=0}^{\infty}(-1)^{s}q^{4s^2+4s}+q^{-2}\Big )\\
&=\frac{1}{1+q^2}\Big ( 1+2q^2\sum_{s=0}^{\infty}(-1)^{s}q^{4s^2+4s}\Big ).\qedhere
\end{align*}

\end{proof}

\section{Miscellaneous examples}\label{section:signFlips}

We recall the identities
\begin{align}
f_{3,3,1}(q^3,q,q)&=J_{1,4}J_{6,12},\\
qf_{1,3,1}(-q,-q^2,-q)&=J_{1,4}V_0(q),
\end{align}
where the first identity comes from \cite[(0.7)]{P1}, and the second identity comes from $V_0(q)$, which is an eighth order mock theta function \cite{GM}, \cite[(5.41)]{HM}:
\begin{equation}
V_{0}(q):=-1+2\sum_{n\ge 0}\frac{q^{n^2}(-q;q^2)_n}{(q;q^2)_{n}}
=-q^{-1}m(1,q^8,q)-q^{-1}m(1,q^8,q^3).
\end{equation}

We also recall a function from the lost notebook studied by Andrews, Dyson, and Hickerson \cite{ADH} and Cohen \cite{C}:
\begin{align}
\sigma(q)&:=1+\sum_{n=1}^{\infty}\frac{q^{n(n+1)/2}}{(-q;q)_n}
=\sum_{n=0}^{\infty}(-1)^{n+j}q^{n(3n+1)/2}(1-q^{2n+1})\sum_{j=-n}^{n}q^{-j^2}.
\end{align}
A straightforward change of variables yields
\begin{align}
\sigma(q)&=g_{1,5,1}(-q,-q,q)-q^2g_{1,5,1}(-q^4,-q^4,q),\\
&=g_{1,3,3}(-q,q^2,q)-qg_{1,3,3}(-q^3,q^4,q).
\end{align}
A function with similar properties was studied by Corson, et al \cite{CFLZ} and Andrews \cite[Theorem 3]{A2009}:
\begin{equation}
\sum_{n=0}^{\infty}\Delta(n)(-q)^n
=\sum_{n=0}^{\infty}\frac{(-1)^nq^{n(n+1)/2}(q;q)_n}{(-q)_n}
\end{equation}
where \cite[Theorem 4]{A2009}
\begin{align}
\sum_{n=0}^{\infty}\Delta(n)q^n&=\sum_{n=0}^{\infty}q^{n(2n+1)}(1+q^{2n+1})\sum_{j=-n}^{n}q^{-j^2}\\
&=g_{1,2,2}(-q^2,-q^3,q^2)+qg_{1,2,2}(-q^4,-q^5,q^2),
\end{align}
where the last equality follows from a change of variables.

Simple sign flips relate the above functions to false theta functions.  We will prove
\begin{theorem} \label{theorem:signFlips1} We have
{\allowdisplaybreaks \begin{align}
g_{3,3,1}(q^3,q,q)
&=\frac{1}{2}\sum_{s=-\infty}^{\infty}\sg(s)(-1)^sq^{\binom{s+1}{2}},\\
g_{1,3,1}(-q,-q^2,-q)
&=\frac{1}{2}\sum_{r=-\infty}^{\infty}\sg(r)q^r(-q)^{\binom{r}{2}}.
\end{align}}%
\end{theorem}

\begin{theorem}\label{theorem:ADH-flips} We have
{\allowdisplaybreaks \begin{align}
g_{1,3,3}(-q,q^2,q)+qg_{1,3,3}(-q^3,q^4,q)
&=\sum_{s=-\infty}^{\infty}\sg(s)(-1)^sq^{s(3s+1)/2},\label{equation:sigma-flip}\\
g_{1,2,2}(-q^{2},-q^{3},q^2)-qg_{1,2,2}(-q^4,-q^5,q^2)
&=-q\sum_{s\in\mathbb{Z}}\sg(s)q^{2s^2+3s}\label{equation:CFLZ-flip}.
\end{align}}%
\end{theorem}

\begin{proof}[Proof of Theorem \ref{theorem:signFlips1}]
Specializing (\ref{equation:g-shift}) with $(\ell,k)=(1,0)$, we have
\begin{align*}
g_{3,3,1}(q^3,q,q)&=-q^3g_{3,3,1}(q^6,q^4,q)+\sum_{s\in\mathbb{Z}}\sg(s)(-1)^sq^{\binom{s+1}{2}}\\
&=-g_{3,3,1}(q^3,q,q)+\sum_{s\in\mathbb{Z}}\sg(s)(-1)^sq^{\binom{s+1}{2}},
\end{align*}
where in the last line we used (\ref{equation:g-flip}).  The first identity then follows.  

Specializing (\ref{equation:g-shift}) with $(\ell,k)=(0,1)$ we have
\begin{align*}
g_{1,3,1}(-q,-q^2,-q)&=q^{2}g_{1,3,1}(-q^4,-q^3,-q)+\sum_{r=-\infty}^{\infty}\sg(r)q^r(-q)^{\binom{r}{2}}
\\
&=-g_{1,3,1}(-q,-q^2,-q)+\sum_{r=-\infty}^{\infty}\sg(r)q^r(-q)^{\binom{r}{2}},
\end{align*}
where for the last equality we used (\ref{equation:g-flip}), and the second identity follows.
\end{proof}

In order to prove Theorem \ref{theorem:ADH-flips}, we need the following proposition
\begin{proposition} \label{proposition:ADH-gform}We have
{\allowdisplaybreaks \begin{align}
g_{1,3,3}(-q,q^2,q)
&=(-1)^kq^{\ell}q^{2k}q^{\binom{\ell}{2}+3\ell k+ 3\binom{k}{2}}g_{1,3,3}(-q^{\ell+3k+1},q^{3\ell + 3k+2},q)
\label{equation:ADH-first}\\
&\ \    +\sum_{r=0}^{\ell-1}q^{\binom{r}{2}+r} \sum_{s=-\infty}^{\infty}\sg(s)(-1)^sq^{s(3s+1+6r)/2}\notag \\
& \  \ 
-\sum_{s=0}^{k-1}(-1)^sq^{3\binom{s}{2}+2s} \sum_{r=1}^{6s}q^{r(r-1-6s)/2} 
 -2\sum_{r=0}^{\ell-1}\sum_{s=0}^{k-1}(-1)^sq^{\binom{r}{2}+3rs+3\binom{s}{2}+r+2s}, \notag\\
 g_{1,2,2}(-q^4,-q^5,q^2)&=q^{4\ell}q^{5k}q^{2\binom{\ell}{2}+4\ell k+ 4\binom{k}{2}}g_{1,2,2}(-q^{2\ell+4k+4},-q^{4\ell + 4k+5},q^2)\label{equation:ADH-second} \\
& \ \  +\sum_{r=0}^{\ell-1}q^{r^2+3r}\sum_{s\in\mathbb{Z}}\sg(s)q^{s(2s+3+4r)}
  -\sum_{s=0}^{k-1}q^{2s^2+3s}\sum_{r=1}^{4s+2}q^{r(r-3-4s)}\notag  \\
&  \ \ -2\sum_{r=0}^{\ell-1}\sum_{s=0}^{k-1}q^{4s}q^{5s}q^{2\binom{r}{2}+4rs+4\binom{s}{2}}.\notag
\end{align}}%
\end{proposition}

\begin{proof}[Proof of Proposition \ref{proposition:ADH-gform}]
We prove the first identity.  The proof for the second identity is similar.  We compute
{\allowdisplaybreaks \begin{align*}
g_{1,3,3}&(-q,q^2,q)\\
&=q^{\ell}(-q^2)^kq^{\binom{\ell}{2}+3\ell k+ 3\binom{k}{2}}g_{1,3,3}(-q^{\ell+3k+1},q^{3\ell + 3k+2},q)\\
&\ \ \ \ \  +\sum_{r=0}^{\ell-1}q^{r}q^{\binom{r}{2}}\Big [ \sum_{s=0}^{\infty}(-1)^sq^{s(3s+1+6r)/2}
-\sum_{s=1}^{\infty}(-1)^sq^{s(3s-1-6r)/2}\Big ]\\
& \ \ \ \ \ 
+\sum_{s=0}^{k-1}(-1)^sq^{2s}q^{3\binom{s}{2}}\Big [ \sum_{r=0}^{\infty}q^{r(r+1+6s)/2}
-\sum_{r=1}^{\infty}q^{r(r-1-6s)/2} \Big ]  \\
& \ \ \ \ \ -2\sum_{r=0}^{\ell-1}\sum_{s=0}^{k-1}(-1)^sq^{r}q^{2s}q^{\binom{r}{2}+3rs+3\binom{s}{2}} \\
& =q^{\ell}(-q^2)^kq^{\binom{\ell}{2}+3\ell k+ 3\binom{k}{2}}g_{1,3,3}(-q^{\ell+3k+1},q^{3\ell + 3k+2},q)\\
&\ \ \ \ \  +\sum_{r=0}^{\ell-1}q^{r}q^{\binom{r}{2}}\Big [ \sum_{s=0}^{\infty}(-1)^sq^{s(3s+1+6r)/2}
-\sum_{s=1}^{\infty}(-1)^sq^{s(3s-1-6r)/2}\Big ]\\
& \ \ \ \ \ 
-\sum_{s=0}^{k-1}(-1)^sq^{2s}q^{3\binom{s}{2}} \sum_{r=1}^{6s}q^{r(r-1-6s)/2} 
 -2\sum_{r=0}^{\ell-1}\sum_{s=0}^{k-1}(-1)^sq^{r}q^{2s}q^{\binom{r}{2}+3rs+3\binom{s}{2}}. \qedhere
\end{align*}}%
\end{proof}

\begin{proof}[Proof of Theorem \ref{theorem:ADH-flips}]
To prove (\ref{equation:sigma-flip}), we specialize (\ref{equation:ADH-first}) with $(\ell,k)=(1,0)$.  We have
{\allowdisplaybreaks \begin{align*}
g_{1,3,3}(-q,q^2,q)&=qg_{1,3,3}(-q^2,q^5,q)
+\sum_{s=0}^{\infty}(-1)^sq^{s(3s+1)/2} -\sum_{s=1}^{\infty}(-1)^sq^{s(3s-1)/2}\\
&=-qg_{1,3,3}(-q^3,q^4)
+\sum_{s=0}^{\infty}(-1)^sq^{s(3s+1)/2} -\sum_{s=1}^{\infty}(-1)^sq^{s(3s-1)/2},
\end{align*}}
where we have used (\ref{equation:g-flip}).  So one can say
\begin{align}
g_{1,3,3}(-q,q^2,q)+qg_{1,3,3}(-q^3,q^4,q)
&=\sum_{s=0}^{\infty}(-1)^sq^{s(3s+1)/2} -\sum_{s=1}^{\infty}(-1)^sq^{s(3s-1)/2},
\end{align}
and the result follows.  To prove (\ref{equation:CFLZ-flip}), we specialize (\ref{equation:ADH-second}) with $ (\ell,k)=(1,0)$.  We find
{\allowdisplaybreaks \begin{align*}
g_{1,2,2}(-q^4,-q^5,q^2)&=q^{4}g_{1,2,2}(-q^{6},-q^{9},q^2)
 +\sum_{s\in\mathbb{Z}}\sg(s)q^{s(2s+3)}\\
 &=q^{-1}g_{1,2,2}(-q^{2},-q^{3},q^2)
 +\sum_{s\in\mathbb{Z}}\sg(s)q^{s(2s+3)},
\end{align*}}%
where we have used (\ref{equation:g-flip}).  Hence
\begin{equation}
g_{1,2,2}(-q^{2},-q^{3},q^2)-qg_{1,2,2}(-q^4,-q^5,q^2)
=-q\sum_{s\in\mathbb{Z}}\sg(s)q^{s(2s+3)},
\end{equation}
and the result follows.
\end{proof}

\section*{Acknowledgements}
This research was supported by Ministry of Science and Higher Education of the Russian Federation, agreement No. 075-15-2019-1619, and by the Theoretical Physics and Mathematics Advancement Foundation BASIS, agreement No. 20-7-1-25-1.


\begin{thebibliography}{999999}


\bibitem{ADH} G. E. Andrews, F. J. Dyson, D. R. Hickerson, {\em Partitions and indefinite quadratic forms}, Invent. Math. {\bf 91} (1988), no. 3, 391--407.

\bibitem{AH} G. E. Andrews, D. R. Hickerson, {\em Ramanujan's `lost' Notebook.  VII: The sixth order mock theta functions}, Adv. Math. {\bf 89} (1991) 60--105.

\bibitem{AW} G. E. Andrews, O. Warnaar {\em The Bailey transform and false theta functions}, Ramanujan J. {\bf 14} (2007), 173--188.

\bibitem{A2009} G. E. Andrews {\em Partitions with distinct evens}, Advances in Combinatorial Mathematics, Springer, Berlin (2009), 31-37.

\bibitem{CK} S. H. Chan, B. Kim {\em On some double-sum false theta series}, J. Number Theory {\bf 190} (2018), 40--55.

\bibitem{C} H. Cohen, {\em $q$-identities for Maass waveform}, Invent. Math. {\bf 91} (1988), 409--422.

\bibitem{CFLZ} D. Corson, D. Favero, K. Liesinger, S. Zubairy, {\em Characters and $q$-series in $\mathbb{Q}(\sqrt{2})$}, J. Number Theory {\bf 107} (2004), 392--405.

\bibitem{GM} B. Gordon, R. McIntosh, {\em Some eighth order mock theta functions,} J. London Math. Soc. (2) {\bf 62} (2000), pp. 321-335.

\bibitem{He} E. Hecke, {\em \"Uber einen Zusammenhang zwischen elliptischen Modulfunktionen und indefiniten quadratischen Formen}, Mathematische Werke, Vandenhoeck and Ruprecht, G\"ottingen, (1959), pp. 418-427.

\bibitem{H1} D. R. Hickerson, {\em A proof of the mock theta conjectures}, Invent. Math {\bf 94}, (1988), 639--660.

\bibitem{H2} D. R. Hickerson, {\em On the seventh order mock theta functions}, Invent. Math {\bf 94}, (1988), 661-677.

\bibitem{HM} D. R. Hickerson, E. T. Mortenson, {\em Hecke-type double sums, Appell--Lerch sums, and mock theta functions,~I},  Proc. London Math. Soc. (3) {\bf 109} (2014), no. 2, 382--422. 

\bibitem{P1} A. Polishchuk, {\em Indefinite theta series of signature (1,1) from the point of view of homological mirror symmetry}, Adv. Math., {\bf 196} (2005), no. 1, pp. 1-51.

\bibitem{R} L. J. Rogers, {\em Second memoir on the expansion of certain infinite products}, Proc. London Math. Soc., {\bf 25} (1894), pp. 318-343.

\bibitem{W1} S. O. Warnaar, {\em 50 years of Bailey's lemma}, A. Betten, et al. (Eds.), Algebraic Combinatorics and Applications, Springer, Berlin, 2001, 333--347.


\end{thebibliography}
\end{document}